\documentclass[a4paper,11pt,leqno]{amsart}

\usepackage{amsmath,amssymb,amsthm}
\usepackage{hyperref}
\usepackage{graphics}
\usepackage{graphicx}

\usepackage{color}

\newcommand{\R}{\mathbb{R}}
\newcommand{\LL}{L}

\newcommand{\mo}[1]{|#1|}

\theoremstyle{plain}
\newtheorem{theo}{Theorem}[section]
\newtheorem{lemma}[theo]{Lemma}

\newtheorem{rem}[theo]{Remark}

\theoremstyle{definition}

\theoremstyle{remark}

\newtheorem*{rema*}{Remark}

\def\bu{{\bf{u}}}
\def\bv{{\bf{v}}}
\def\ue{u^e}
\def\ve{v^e}
\def\cn{\mathbb{C}}

\begin{document}

\title[Dispersion for Schr\"odinger with {$\delta$ potentials and for $\delta$ trees}]{Dispersion for the Schr\"odinger equation \\  {on the line with multiple Dirac delta potentials and on delta trees}}  

\author[V. Banica]{Valeria Banica}
\address[V. Banica]{Laboratoire de Math\'ematiques et de Mod\'elisation d'\'Evry (UMR 8071)\\ D\'epartement de Math\'ematiques\\ 
Universit\'e d'\'Evry, 23 Bd. de France, 91037 Evry\\ 
France.} 
\email{Valeria.Banica@univ-evry.fr}
\author[L. I. Ignat]{Liviu I. Ignat}
\address[L. I. Ignat]{Institute of Mathematics ``Simion Stoilow'' of the Romanian Academy\\21 Calea Grivitei Street \\010702 Bucharest \\ Romania 
\hfill\break\indent \and
\hfill\break\indent 
Faculty of Mathematics and Computer Science, University of Bucharest, 14 Academiei Str., 010014 Bucharest, Romania.
}

\email{liviu.ignat@gmail.com}
\thanks{V.B. is partially
  supported by the French ANR projects  "R.A.S." ANR-08-JCJC-0124-01 and "SchEq" ANR-12-JS01-0005-01.\\ L. I. is partially supported by   PN-II-ID-PCE-2011-3-0075 
of CNCSIS--UEFISCDI Romania, MTM2011-29306-C02-00, MICINN, Spain, and ERC Advanced Grant FP7-246775 NUMERIWAVES}

\begin{abstract}
In this paper we consider the time dependent one-dimensional Schr\"odinger equation with multiple Dirac delta potentials {of different  strengths}. 
We prove that the classical dispersion property holds under some restrictions on the strengths and on the lengths of the finite intervals. The result is obtained in a more general setting of a Laplace operator on a tree 
with $\delta$-coupling conditions at the vertices. The proof relies on a careful analysis of the properties of the resolvent of the associated Hamiltonian. With respect to the analysis done in \cite{MR2858075} for Kirchhoff conditions, here the resolvent is no longer in the framework of Wiener algebra of almost periodic functions, and its expression is harder to analyze.
\end{abstract}

\maketitle

\section{Introduction}

In this paper we are concerned with the dispersive properties of the Schr\"odinger equation with multiple Dirac delta potentials and more generally for the Schr\"odinger equation on a tree with $\delta$-coupling conditions at the vertices. 

Let us first recall that the linear Schr\"{o}dinger equation on the line 
\begin{equation}\label{sch1}
\left\{\begin{array}{l}
  iu_t(t,x)+ u_{xx}(t,x) = 0,  \,(t,x)\in \R\times\R,\\
  u(0,x)  =u_0(x), \,x\in \R,
    \end{array}\right.
\end{equation}
conserves the $\LL^2$-norm
\begin{equation}\label{energy}
\|e^{it\Delta}u_0\|_{\LL^2(\R)}=\|u_0\|_{\LL^2(\R)}
\end{equation}
 and enjoys the dispersive estimate
\begin{equation}\label{linfty}
\|e^{it\Delta}u_0\|_{L^\infty(\R)}\leq \frac C{\sqrt{|t|}}\|u_0\|_{\LL^1(\R)}, \
\ t\neq 0.
\end{equation}
It is classical to obtain from these two inequalities the well-known space-time Strichartz estimates (\cite{0372.35001},\cite{MR801582}), for $r\geq 2$,
\begin{equation}\label{dsch0001}
    \|e^{it\Delta}u_0\|_{\LL^\frac{4r}{r-2}_t(\R,\,\LL^r_x(\R))}\leq C\|u_0\|_{\LL^2(\R)}.
\end{equation}
These dispersive estimates have been successfully applied to obtain results for the nonlinear Schr\"odinger equation (see for example \cite{1055.35003}, \cite{MR2233925} and the reference therein).

Our general framework in this paper refers to the Dirac's delta Hamiltonian on a tree with {a finite number of vertices}, with the external edges (those that have only one internal vertex as an endpoint) formed by infinite strips. The particular case of a tree with all the internal vertices having degree two will give us a result for the Schr\"odinger equation on the line with several Dirac potentials. Although the later is a corollary of the former, we shall start our presentation by the case of the line. This is motivated by the fact that historically dispersive properties have been studied first in this case (only with one or with two delta Dirac potentials) and that the previous results on graphs concern only star-shaped graphs (with only one vertex), where the proofs are in the same spirit as on the line with one Dirac delta potential.

So we first consider the semigroup $\exp(-it H_\alpha)$ where $H_\alpha$ is a perturbation of the Laplace operator with {$n$} Dirac delta potentials with   real strengths $\{\alpha_j\}_{j=1}^p$
\begin{equation}\label{Halpha}
H_{\alpha}=-\Delta + \sum _{j=1}^p \alpha_j \delta (x-x_j).
\end{equation}

%
 
The spectral properties of the Laplacian with multiple Dirac delta potentials on $\mathbb R^n$ have been extensively studied. Operator $H_\alpha$ has
at most $p$ eigenvalues which are all negative and simple, and there are no eigenvalues in case of positive strengths $\alpha_i>0$. The remaining part of the spectrum is absolutely continuous and 
$\sigma_{ac}(H_\alpha)=[0,\infty)$. We will denote along the paper $P_e$  the $L^2$ projection onto the subspace of the eigenfunctions and by $P$ 
 the projection outside the discrete spectrum.
Regarding the spectral properties of $H_\alpha$ we  refer to \cite[Ch.~II. 2]{MR2105735} and to the references within. The time dependent propagator of the linear Schr\"odinger equation has also been considered in the case of one Dirac delta potential  \cite{MR851483}, \cite{MR975175}, \cite{MR2188549}, \cite{MR2560311}, or one point interactions
  \cite{MR1295192}, \cite{MR2566277}, \cite{MR2457813}, or two symmetric Dirac delta  potentials  \cite{MR2608273}. In particular, in the case of the line with one delta interaction, without sign condition on the strength,  dispersive estimates has been proved but for $e^{-itH_\alpha}P$ (\cite{MR2188549},\cite{MR2560311}). A similar result was proved to hold in case of two point interactions, under a condition on the delta-strength and on the distance between the location of the point interactions (\cite{MR2608273}, see also \cite{AnFe}). 
  Also in \cite{MR2791127} the problem of dispersion for several delta potential has been considered, as well as wave operators bounds from which dispersive estimate can be obtained as a consequence. Here Jost and distorted plane functions are used in spectral formulae. General conditions for the main results to hold are given for general potentials with singularities. In the case of Dirac potentials these are proved to hold for the case of one Dirac potential and for the case of the double delta well potential.
  Concerning the nonlinear Schr\"odinger equation with a Dirac delta potential, standing wave and bound states have been analyzed \cite{MR2379459}, \cite{MR2457813}, \cite{MR2450497}, as well as the time dynamics of solitons \cite{MR2318852}, \cite{MR2342704}, \cite{MR2335125}. 

 For stating our first result concerning the case of several Dirac potentials, we need to introduce the following functions. With the notations in Lemma \ref{dets} in the case when $n_j=2$, we denote $f_p=\det D_p$ and $g_p=\frac{\det\tilde D_p}{\det D_p}$, defined  by recursion as follows:
 $$f_1(\omega)=\frac{2\,\omega+\alpha_1}{\omega+\alpha_1},\,\,\,f_p(\omega)=\frac{2\,\omega+\alpha_p}{\omega+\alpha_p}e^{\omega a_{p-1}}f_{p-1}(\omega)\left(1-\frac{\alpha_p}{2\,\omega+\alpha_p}\,e^{-2\omega a_{p-1}} g_{p-1}(\omega)\right),$$
where
$$g_1(\omega)=\frac{\alpha_1}{n_1\,\omega+\alpha_1},\,\,\,g_p(\omega)=\frac{\frac{\alpha_p}{n_p\,\omega+\alpha_p}-\frac{-2\,\omega+\alpha_p}{2\,\omega+\alpha_p}\,e^{-2\omega a_{p-1}}g_{p-1}(\omega)}{1-\frac{\alpha_p}{2\omega+\alpha_p}\,e^{-2\omega a_{p-1}}g_{p-1}(\omega)}.$$
These functions will appear naturally when computing the resolvent of $H_\alpha$.

\begin{theo}  \label{disp-0}For any  $\{\alpha_j\}_{j=1}^p$ and $\{x_j\}_{j=1}^p$ such that 
\begin{equation}\label{cond-0}
\partial_{\omega}^{p-1} f_p\,_{\scriptscriptstyle{\vert \omega=0}}\neq 0,
\end{equation} 
 the solution of the linear Schr\"odinger equation on the line with multiple delta interactions of strength $\alpha_j$ located at $x_j$ satisfies the
dispersion inequality 
\begin{equation}\label{dispersion-0}
\| e^{-itH_\alpha}Pu_0\|_{L^\infty(\R)}\leq \frac{C}{\sqrt {|t|}}
\| u_0\|_{L^1(\R)},\,\,\, \forall t\neq 0.
\end{equation}
Moreover, in case of positive strengths $\alpha_j>0$, condition \eqref{cond-0} is fulfilled and we have
\begin{equation}\label{dispersion-00}
\| e^{-itH_\alpha}u_0\|_{L^\infty(\R)}\leq \frac{C}{\sqrt {|t|}}
\| u_0\|_{L^1(\R)},\,\,\, \forall t\neq 0.
\end{equation}
\end{theo}

We first notice that in view of the definition of $f_p(\omega)$, condition \eqref{cond-0} is not fulfilled only in a few explicit situations. For instance, if $p=2$, the situations to be avoided are when $x_2-x_1+\frac{\alpha_1+\alpha_2}{\alpha_1\alpha_2}=0$ already used in \cite{MR2608273}.

In the previous works on dispersive estimates for one or two delta Dirac potentials, given the particular structure of the operator $H_\alpha$,  the authors obtain explicit representations of the resolvent and then of $e^{-itH_\alpha}$. However in the general case of multiple delta interactions an explicit representation is not easy to obtain; even in \cite{MR757115}, \cite[Ch.~II.2]{MR2105735} the resolvent is obtained in terms of the inverse of some matrix $D_{n}$ that depends on $\{\alpha_j\}_{j=1}^p$ and on the lengths of the finite segments $\{x_j-x_{j-1}\}_{j=2}^p$. \\

The line setting might be seen as the special case of the equation posed on a simple graph {with $n$ vertices,} with only two edges starting from any vertex and with delta connection conditions at each vertex ($x_0=-\infty$, $x_{p+1}=\infty$)
\begin{equation}\label{}
\left\{
\begin{array}{ll}
iu_t(t,x)+u_{xx}(t,x)=0, & x\in (x_{j-1},x_j),  j=1,\dots, p,\\[10pt]
u_x(t,x_j^+)-u_x(t,x_j^-)=\alpha_j u(x_j), &t>0,\ j=1,\dots,p.
\end{array}
\right. 
\end{equation}

Our second framework refers to the Dirac's delta Hamiltonian $H^\Gamma_\alpha$ on   tree $\Gamma=(V,E)$ with {a finite number of vertices $V$}, with the external  edges (thats that have only one internal vertex as an endpoint) formed by infinite strips. We consider the linear Schr\"odinger equation in the  case of a  tree $\Gamma$, with delta conditions of non necessarily equal strength at the vertices
\begin{equation}\label{schHdelta}
\left\{\begin{array}{l}
  i\bu _t(t,x)= H^\Gamma_\alpha \bu (t,x) ,  \,(t,x)\in \R\times\Gamma,\\
  \bu(0,x)  =\bu_0(x), \,x\in \Gamma.
    \end{array}\right.
\end{equation}
 The presentation of  operator $H^\Gamma_\alpha$ will be given in full details in Sec. \ref{graphs}. Let us just
say here that $H^\Gamma_\alpha$ acts on a function $\bf u$ on a graph as $-\partial_{xx}$ on each restriction of $\bf u$ on an edge of the tree and that its domain are functions $\bf u$ for which $\delta$-coupling conditions must be fulfilled. The $\delta$-coupling conditions are continuity condition for the function $\bf u$ and a $\delta-$transmission condition at the level of its first derivative at all internal vertices $v$:  $$\sum_{e\in E_v} \partial_n \bu(v)=\alpha(v)\bu(v).$$
The operator $H^\Gamma_\alpha$ shares the same properties of $H_\alpha$ above: only a finite number of negative eigenvalues, and no eigenvalues for positive strengths, and $\sigma_{ac}(H^\Gamma_\alpha)=[0,\infty)$. These properties follow as in \cite[Ch.~II.2]{MR2105735}.

The dispersion inequality for equation \eqref{schHdelta} was proved in \cite{MR2858075} (see also \cite{MR2684310}) for the case of Kirchhoff's connection condition on trees, i.e. $\alpha(v)=0$ for all internal vertices of the  tree. The case of $\delta$ and $\delta'$ coupling on a star shaped tree (i.e. only one vertex) has been considered in \cite{MR2804557}, where the main result concerns the time evolution of a fast soliton for the nonlinear equation, in the spirit of \cite{MR2318852}. Finally, we mention that for the stationary nonlinear equation, the study of bound states on a star shaped tree with delta conditions has been analyzed in a series of papers \cite{adami1}, \cite{adami2}, \cite{adami3}, \cite{adami4}.

The main result of this paper is the following, involving the expression of a determinant function $\det D_p(\omega)$ defined by recursion in Lemma \ref{dets}.
\begin{theo}  \label{disp}
Let us consider a tree $\Gamma=(V,E)$ with $p$ vertices. If the strengths at the vertices and the lengths of the finite edges are such that 
\begin{equation}\label{cond}
\partial_{\omega}^{(p-1)}\det D_p\,_{\scriptscriptstyle{\vert \omega=0}}\neq 0,
\end{equation}
then 
the solution of the linear Schr\"odinger equation on a tree with delta connection conditions  satisfies the
dispersion inequality 
\begin{equation}\label{dispersion}
\| e^{-itH^\Gamma _\alpha}P\bu_0\|_{\LL^\infty(\Gamma)}\leq \frac{C}{\sqrt {|t|}}
\| \bu_0\|_{\LL^1(\Gamma)},\,\,\, \forall t\neq 0.
\end{equation}
Moreover, in case of positive strengths $\alpha_j>0$, condition \eqref{cond} is fulfilled and we have
\begin{equation}\label{dispersion-bis}
\| e^{-itH^\Gamma _\alpha}\bu_0\|_{\LL^\infty(\Gamma)}\leq \frac{C}{\sqrt {|t|}}
\| \bu_0\|_{\LL^1(\Gamma)},\,\,\, \forall t\neq 0.
\end{equation}
\end{theo}

The proof {of Theorem \ref{disp}} uses elements from \cite{MR2049025}, \cite{MR2858075}, \cite{gavrus} in an appropriate way related to the delta connection conditions on the tree. The starting point consists in writing the solution in terms of the resolvent of the Laplacian, which in turn is determined by recursion on the number of vertices. With respect to the previous works with Kirchhoff conditions, the novelty here is that we are not any longer in the framework of the almost periodic Wiener algebra of functions, and that the expression of the resolvent is harder to analyse.

The linear solution $e^{-itH^\Gamma _\alpha}\bu_0$ will be shown to be a combination of oscillatory integrals, that becomes more and more involved as the number of vertices of the tree grows. We do not have any more that $e^{-itH^\Gamma _\alpha}\bu_0$ is a summable superposition of solutions of the linear Schr\"odinger equation on the line, as for Kirchhoff conditions in \cite{MR2858075}.

Theorem \ref{disp-0} follow from Theorem \ref{disp} by considering the particular case of a tree $\Gamma$ with all the internal vertices having degree two.\\

As classically noticed (\cite{Ra,JeKa,JoSoSo,RoSc,GoSc}), one can expect dispersion in absence of eigenvalues and of zero resonances. In the $\delta-$coupling case the non-generic condition \eqref{cond-0} for $p=2$ is precisely in link with the presence of a zero resonance (see formula (2.1.29) of Ch. II.2.1 in \cite{MR2105735}). So one might expect that in the absence of eigenvalues the dispersion holds generically, even for more general coupling. We shall give in  Appendix \ref{gen} some sufficient conditions to obtain dispersion for general couplings.\\

Finally, we note that in the presence of eigenfunctions, the dispersion estimate cannot be valid globally in time.  
Denoting by $H$ either $H_\alpha$ or $H^\Gamma_\alpha$, the general classical $TT^*$ argument and Christ-Kiselev lemma allow to infer global in time Strichartz estimates as on $\mathbb R$  for $e^{-itH}P$ the dispersive part of $e^{-itH} $ (see for instance the short proof of Theorem 2.3 in \cite{MR2233925}). This together with the regularity of the eigenfunctions of the operator $H$ give us the following result:
\begin{theo}\label{strichartz}
Let $T>0$ and let $(q,r)$ and $(q',r')$ be two $1-$admissible couples, i.e. $4\leq q\leq \infty, 2\leq r\leq \infty$ and $\frac 2q+\frac 1r=\frac 12$. For any $\alpha\geq 1$, there exists a constant $C>0$ such that homogeneous Strichartz estimates
$$\|e^{-itH} \bu_0\|_{L^{q}((0,T), L^r(\Gamma))}\leq C(\|\bu_0\|_{L^2(\Gamma)}+T^{1/q}\|\bu_0\|_{L^\alpha(\Gamma)}),$$
and the inhomogeneous Strichartz estimates
$$\|\int_0^t e^{-i(t-s)H} {\bf{F}} (s)\,ds\|_{L^{q}((0,T), L^r(\Gamma))}\leq C(\|{\bf{F}} \|_{L^{\tilde q'}((0,T),L^{\tilde r'}(\Gamma))}+
T^{1/q}\|{\bf{F}} \|_{L^{ 1}((0,T),L^\alpha(\Gamma))}),$$
hold. Here $x'$ stands for the conjugate of $x$, defined by $\frac 1x+\frac 1{x'}=1$.
\end{theo}
We shall give in Appendix \ref{appStrich} a proof inspired from \cite{MR2560311}. 
As a typical result for the nonlinear Schr\"odinger equation based on the Strichartz estimates one obtains the global in time wellposedness for subcritical $L^2(\Gamma)$ solutions: 
\begin{theo}\label{tree.nse}
 Let $p\in (0,4)$. For any $\bu_0\in \LL^2(\Gamma)$ there exists a unique solution
$$\displaystyle \bu\in C(\R,\LL^2(\Gamma))\cap \bigcap _{(q,r) 1-adm.} \LL^{q}(\R,\LL^r(\Gamma)),$$
of the nonlinear Schr\"odinger equation
\begin{equation}\label{eq.tree.non}
\left\{
\begin{array}{ll}
i\bu_t +H \bu\pm|\bu|^{p}\bu=0,& t\neq 0,\\[10pt]
 \bu(0)=\bu_0,& t=0.
\end{array}
\right.
\end{equation}
Moreover, the $\LL^2(\Gamma)$-norm of $\bu$ is conserved along the time
$$\|\bu(t)\|_{\LL^2(\Gamma)}=\|\bu_0\|_{\LL^2(\Gamma)}.$$
\end{theo}

Local in time existence with lifespan depending on the $L^2$ size of the initial data follows from a classical fixed point argument as on $\mathbb R$ (see for instance Proposition 3.15 in \cite{MR2233925}). The extension to global solutions is obtained from the conservation of the $L^2(\Gamma)$-norm that in turn follows by taking the imaginary part of the equation \eqref{eq.tree.non} multiplied by $\overline \bu$, and integrating on $\Gamma$.\\

%

The paper is organized as follows. In the next section we introduce the framework of the Laplacian analysis on a graph. In \S 3 we give the proof of Theorem \ref{disp}. In the first Appendix we show how the conditions of the theorems are fulfilled for positive strengths of interactions. The second Appendix contain the proof of Theorem \ref{strichartz}. In the last Appendix we shall describe the approach for general coupling conditions.\\

{\bf {Aknowledgements.}} The authors are grateful to the referee for the remarks and questions that improved the presentation of this paper.

\section{Preliminaries on graphs and $\delta$-coupling}\label{graphs}

In this section we present some generalities about metric graphs and introduce the Dirac's delta Hamiltonian  $H^\Gamma_\alpha$
on such structure.  
More general type of self-adjoint operators, $\Delta(A,B)$, \,have been considered in \cite{MR2277618}, \cite{MR1671833}. 
We collect here some basic facts on metric graphs and on some operators that could be defined on such structure
 \cite{MR2459876}, \cite{MR2042548}, \cite{MR2148631}, \cite{MR2277618}, \cite{2006AdPhy..55..527G}, \cite{MR2883421}.

Let $\Gamma=(V,E)$ be a graph where $V$ is a set of vertices and $E$ the set of edges. 
For each $v\in V$ we denote by $E_v=\{e\in E: v\in e\}$ the set of edges branching from $v$.
We assume that  $V$ is connected and 
the degree of each vertex $v$ of $\Gamma$ is finite:
$d(v)=|E_v|<\infty.$
The edges could be of finite length and then their ends are vertices of $V$ or they have infinite length and then we assume that each infinite edge is a ray with a single vertex belonging to $V$ (see \cite{MR2459876} for more details on graphs with infinite edges). The vertices are called internal if $d(v)\geq 2$ or external if $d(v)=1$. 
In this paper we will assume that there are not external vertices.

We fix an orientation of $\Gamma$ and for each oriented edge $e$, we denote by $I(e)$ the initial vertex and by $T(e)$ the terminal one. Of course in the case of infinite edges we have only initial vertices.

We identify every edge $e$ of $\Gamma$ with an interval $I_e$, where $I_e=[0,l_e]$ if the edge is finite and $I_e=[0,\infty)$ if the edge is infinite. This identification introduces a coordinate $x_e$ along the edge $e$. In this way $\Gamma$ is a metric space and is often named metric graph  \cite{MR2459876}. 

Let $v$ be a vertex of $V$ and $e$ be an edge in  $E_v$. We set for finite edges $e$
$$j(v,e)=\left\{
\begin{array}{lll}
0& \text{if} &v=I(e), \\[10pt]
l_e& \text{if} & v=T(e) 
\end{array}
\right.
$$
and
$$j(v,e)=0,\ \text{if}\ v=I(e)$$
for infinite edges.

We identify any function $\bu$ on $\Gamma$ with a collection $\{\ue\}_{e\in E}$ of functions $\ue$ defined on the edges  $e$ of $\Gamma$. Each $\ue$ can be considered as a function on the interval $I_e$. In fact, we use the same notation $\ue$ for both the function on the edge $e$ and the function on the interval $I_e$ identified with $e$.
For a function $\bu:\Gamma\rightarrow \cn$,  $\bu=\{u^e\}_{e\in E}$,  we denote by $f(\bu):\Gamma\rightarrow \cn$ the family 
$\{f(u^e)\}_{e\in E}$, where  $f(u^e):e\rightarrow\cn$.

A function $\bu=\{\ue\}_{e\in E}$ is continuous if and only if $\ue$ is continuous on $I_e$ for every $e\in E$, and moreover, is continuous at the vertices of $\Gamma$:
$$\ue(j(v,e))=u^{e'}(j(v,e')), \quad \forall \ e,e'\in E_v,{\quad \forall \ v\in V.}$$ 

The space $\LL^p(\Gamma)$, $1\leq p<\infty$ consists of all functions   $\bu=\{u_e\}_{e\in E}$ on $\Gamma$ that belong to $\LL^p(I_e)$
for each edge $e\in E$ and 
$$\|\bu\|_{\LL^p(\Gamma)}^p=\sum _{e\in E}\|u^e\|_{\LL^p(I_e)}^p<\infty.$$
Similarly, the space $\LL^\infty(\Gamma)$ consists of all functions that belong to $\LL^\infty(I_e)$ for each edge $e\in E$ and
$$\|\bu \|_{\LL^\infty(\Gamma)}=\sup _{e\in E}\|u^e\|_{\LL^\infty(I_e)}<\infty.$$
The Sobolev space $H^m(\Gamma)$, $m\geq 1$ an integer, consists in all continuous functions on $\Gamma$ that  belong to
$H^m(I_e)$ for each $e\in E$ and 
$$\|\bu \|_{H^m(\Gamma)}^2=\sum _{e\in E}\|u^e\|_{H^m(e)}^2<\infty.$$
The above spaces are  Hilbert spaces with the inner products
$$(\bu,\bv)_{\LL^2(\Gamma)}=\sum _{e\in E}(\ue,\ve)_{\LL^2(I_e)}=\sum _{e\in E}\int _{I_e}\ue(x)\overline{\ve}(x)dx$$
and
$$(\bu,\bv)_{H^m(\Gamma)}=\sum _{e\in E}(\ue,\ve)_{H^m(I_e)}= \sum _{e\in E}\sum _{k=0}^m\int _{I_e} \frac{d^k\ue}{dx^k} \overline{\frac{d^k\ve}{dx^k} }dx.$$
We now define the normal exterior derivative of a function $\bu=\{u^e\}_{e\in E}$  at the endpoints of the edges.
For each $e\in E$ and $v$ an endpoint of $e$ we consider the normal derivative of the restriction of $\bu$ to the edge $e$ of $E_v$ evaluated at $j(v,e)$ to be defined by:
$$\frac {\partial u^e}{\partial n_e}(j(v,e))=
\left\{
\begin{array}{lll}
-u_x^e(0^+)&\text{if}& j(v,e)=0, \\[10pt]
u_x^e(l_e^-)& \text{if}&j(v,e)=l_e  .
\end{array}
\right.
$$

We now introduce $H^\Gamma_\alpha$. It generalizes the classical Dirac's delta interactions with strength  parameters \eqref{Halpha}.
The Dirac's delta Hamiltonian is defined on the domain
\begin{equation}\label{domHdelta}
D(H^\Gamma_\alpha)=\Big\{\bu\in H^2(\Gamma),\  \sum _{e\in E_v} \frac {\partial u^e}{\partial n_e}(j(v,e)) =  \alpha (v) \bu(v), \quad \forall v\in V \Big \}.
\end{equation}
Operator $H^\Gamma_\alpha$ acts as following, for any $\bu=\{u^e\}_{e\in E}$ 
$$(H^\Gamma_\alpha \bu)(x)=-u_{xx}^e(x),\quad x\in I_e,\ e\in E.$$
The quadratic form associated to $H^\Gamma_\alpha$ is defined on $H^1(\Gamma)$ and it is given by
$$\mathcal{E}^\Gamma_\alpha (u)= \sum _{e\in E}\int _{I_e}|\ue_x(x)|^2dx +\sum _{v\in V} \alpha(v)|\bu(v)|^2.$$
The case when all strengths vanish corresponds to the Kirchhoff coupling analyzed in \cite{MR2858075}.

Finally, let us mention that there are other  coupling conditions (see \cite{MR1671833}) which allow to define a ``Laplace" operator on a metric graph. To be more precise, let us consider an operator that acts on functions on the graph $\Gamma$ as the second derivative $\frac {d^2}{dx^2}$, and its domain consists in all functions $\bf{u}$ that belong to the Sobolev space
$H^2(e)$ on each edge $e$ of $\Gamma$ and satisfy the following boundary condition at the vertices:
\begin{equation}\label{con.1}
A(v){ \bf {u}} (v)+B(v){\bf {u}}'(v)=0 \quad \text{for each vertex} \ v.
\end{equation}
Here ${\bf {u}}(v)$ and ${\bf {u}}'(v)$ are correspondingly the vector of values of $\bf{u}$ at $v$ attained from directions of different edges converging at $v$ and the vector of derivatives at $v$ in the outgoing directions.
For each vertex $v$ of the tree we assume that  matrices $A(v)$ and $B(v)$ are of size $d(v)$ and satisfy the following two conditions
\begin{enumerate}
\item the joint matrix $(A(v), B(v))$ has maximal rank, i.e.  $d(v)$,\\
\item $A(v)B(v)^T=B(v)A(v)^T$. 
\end{enumerate} 

Under those assumptions it has been proved in \cite{MR1671833} that the considered operator, denoted by $\Delta(A,B)$, is self-adjoint. The case considered in this paper, the $\delta$-coupling, corresponds to the matrices
$$A(v)=\left(\begin{array}{cccccc}1 & -1 & 0 & \dots & 0 & 0 \\0 & 1 & -1 & \dots & 0 & 0 \\0 & 0 & 1 & \dots & 0 & 0 \\\vdots & \vdots & \vdots &  & \vdots & \vdots \\0 & 0 & 0 & \vdots & 1 & -1 \\0 & 0 & 0 & \vdots & 0 & -\alpha(v)\end{array}\right),\ 
B(v)=\left(\begin{array}{cccccc}0 & 0 & 0 & \dots & 0 & 0 \\0 & 0 & 0 & \dots & 0 & 0 \\0 & 0 & 0 & \dots & 0 & 0 \\\vdots & \vdots & \vdots &  & \vdots & \vdots \\0 & 0 & 0 & \dots & 0 & 0 \\1 & 1 & 1 & \dots & 1 & 1\end{array}\right).$$
More examples of matrices satisfying the above conditions are given in \cite{MR1671833, MR2459885}.

\section{Proof of Theorem \ref{disp}}\label{const}

We shall use a description of the solution of the linear Schr\"odinger equation in terms of the resolvent. 
For $\omega>0$ such that $-\omega^2$ is not an eigenvalue, let $R_\omega$ 
be the resolvent of the Laplacian on a tree
$$R_{\omega}{\bf u_0}=(H^\Gamma_\alpha+\omega ^2I)^{-1}{\bf u_0}.$$ 

Before starting let us  choose an orientation on  tree $\Gamma$. Let us choose an internal vertex $\mathcal{O}$. This will be the root of the tree and the initial 
vertex for all the edges that branch from it. This procedure introduces an orientation for all the edges staring from $\mathcal{O}$. 
For the other endpoints of the edges belonging to $E_\mathcal{O}$   we repeat
the above procedure and inductively we construct an orientation on $\Gamma$.

\subsection{The structure of the resolvent}\label{sect:res}

In order to obtain the expression of the resolvent second-order equations 
$$(R_\omega \bold u_0)''=\omega^2 R_\omega \bold u_0-\bold u_0$$ 
must be solved on each edge of the tree together with coupling conditions at each vertex. Then, on each edge parametrized by $I_e$, for $x\in I_e$, since $\omega\neq 0$,
\begin{equation}\label{res}
R_\omega \bold u_0(x)=c_e\,e^{\omega x}+\tilde c_e\,e^{-\omega
x}+\frac{t_e(x,\omega)}{\omega}, 
\end{equation}
with
$$t_e(x,\omega)=\frac{1}{2}\int_{I_e}\bold u_0(y)\,e^{-\omega \mo{x-y}}dy.$$

Since $R_\omega \bold u_0$ belongs to $\LL^2(\Gamma)$ the coefficients $c$'s are zero on the  infinite edges $e\in\mathcal E$, parametrized by $[0,\infty)$. If we denote by $\mathcal I$ the set of internal edges, we have $2|\mathcal I|+|\mathcal E|$ coefficients. The delta conditions of continuity of $R_\omega \bold u_0$ and 
of transmission of $(R_\omega \bold u_0)'$ at the vertices of the tree give the system of 
equations on the coefficients. We have the same number of equations as the number of unknowns. We denote $D_{\Gamma_p}(\omega)$ the matrix of the system, where $p$ stands for the number of vertices of the tree, and by $T_{\Gamma_p}(\omega)$ the column of the free terms in the system.

Therefore the resolvent $R_\omega \bold u_0(x)$ on an edge $I_e$ is   
\begin{equation}\label{ressum}
R_\omega \bold u_0(x)=\frac{\det M^{c_e}_{\Gamma_p}(\omega)}{\det D_{\Gamma_p}(\omega)}\,e^{\omega x}
+\frac{\det M^{\tilde c_e}_{\Gamma_p}(\omega)}{\det D_{\Gamma_p}(\omega)}\,e^{-\omega x}+\frac{t_e(x,\omega)}{\omega},
\end{equation}
where $ M^{c_e}_{\Gamma_p}(\omega)$ and $ M^{\tilde c_e}_{\Gamma_p}(\omega)$ are obtained from $D_{\Gamma_p}(\omega)$ by replacing the column corresponding to the unknown $c_e$, and respectively $\tilde c_e$ by the column of the free terms $T_{\Gamma_p}(\omega)$.



\subsection{The expression of $\det D_{\Gamma_p}(\omega)$}
In view of the form \eqref{res} of the resolvent, we obtain on an edge $I_e$
\begin{equation}\label{deltacond}
R_\omega \bold u_0(0)=c_e+\tilde c_e+\frac{t_e(0,\omega)}{\omega},
\end{equation}
$$(R_\omega \bold u_0)'(0)=c_e\,\omega-\tilde c_e\,\omega+t_e(0,\omega),$$
and in case $I_e$ is parametrized by $[0,a]$ with $a<\infty$,
\begin{equation}\label{deltaconda}
R_\omega \bold u_0(a)=c_e\,e^{\omega a}+\tilde c_e\,e^{-\omega
a}+\frac{t_e(0,\omega)}{\omega},
\end{equation}
$$(R_\omega \bold u_0)'(a)=c_e\,\omega\, e^{\omega a}-\tilde c_e\,\omega\,e^{-\omega
a}-t_e(a,\omega).$$

\subsubsection{The star-shaped tree case}\label{star} In the case of a single vertex and $n_1\geq 2$ edges $I_j$, $1\leq j\leq n_1$, parametrized by $[0,\infty)$ we have only the coefficient $\tilde c_j$ on each edge $I_j$ since all $c_j$ vanish. The delta conditions are continuity of the resolvent at the vertex, together with the fact that the sum of the first derivatives must be equal to $\alpha$ times the value of the resolvent at the vertex
$$(R_\omega \bold u_0)_j(0)=(R_\omega \bold u_0)_1(0), \,\,\sum_{1\leq j\leq n_1}(R_\omega \bold u_0)'_j(0)=\alpha_1 \,(R_\omega \bold u_0)_1(0).$$
From \eqref{deltacond} we obtain as matrix for the system of $\tilde c$'s
$$D_{\Gamma_1}(\omega)=\left(\begin{array}{cccccccc}1 & -1 & &&&&&\\  & 1 & -1&&&&&\\ &&.&.&&&&\\&&&.&.&&&\\&&&&.&.&&\\&&&&&1&-1&\\&&&&&&1&-1\\1&\frac{\omega}{\omega+\alpha_1}&\frac{\omega}{\omega+\alpha_1}&
.&.&\frac{\omega}{\omega+\alpha_1}&\frac{\omega}{\omega+\alpha_1}&\frac{\omega}{\omega+\alpha_1} \end{array}\right),$$
and as a free term column 
$$T_{\Gamma_1}(\omega)=\left(\begin{array}{cccccccc}
\frac{t_2(0,\omega)-t_1(0,\omega)}{\omega}\\
...\\
\frac{t_{n_1}(0,\omega)-t_{n_1-1}(0,\omega)}{\omega}\\
\,\\
\frac{\omega-\alpha_1}{\omega+\alpha_1}\frac{t_1(0,\omega)}{\omega}+\frac{\omega}{\omega+\alpha_1}\sum_{2\leq j\leq n_1}\frac{t_j(0,\omega)}{\omega} \end{array}\right).$$
By developing $\det D_{\Gamma_1}(\omega)$ with respect to its last column, we obtain by recursion that
$$\det D_{\Gamma_1}(\omega)=\frac{n_1\,\omega+\alpha_1}{\omega+\alpha_1}.$$
Thus $\det D_{\Gamma_1}$ does no vanish on the imaginary axis and $\omega R_\omega\bold u_0$ can be analytically continued in a region containing the imaginary axis.

We introduce here the matrix $\tilde D_{\Gamma_1}(\omega)$ which is the matrix of the coefficients of the resolvent, if on the last edge $I_{n_1}$ we should have $c_{n_1}e^{\omega x}$ instead of $\tilde c_{n_1}e^{-\omega x}$. This changes only the $(n_1,n_1)$-entry of $D_{\Gamma_1}(\omega)$ in $-\frac{\omega}{\omega+\alpha_1}$ instead of $\frac{\omega}{\omega+\alpha_1}$,
$$\tilde D_{\Gamma_1}(\omega)=\left(\begin{array}{cccccccc}1 & -1 & &&&&&\\  & 1 & -1&&&&&\\ &&.&.&&&&\\&&&.&.&&&\\&&&&.&.&&\\&&&&&1&-1&\\&&&&&&1&-1\\1&\frac{\omega}{\omega+\alpha_1}&\frac{\omega}{\omega+\alpha_1}&
.&.&\frac{\omega}{\omega+\alpha_1}&\frac{\omega}{\omega+\alpha_1}&-\frac{\omega}{\omega+\alpha_1} \end{array}\right).$$ 
Moreover, the free term column remains the same for this new system. We have again by recursion
$$\det \tilde D_{\Gamma_1}(\omega)=\frac{(n_1-2)\,\omega+\alpha_1}{\omega+\alpha_1}.$$

\subsubsection{The general tree case}
Any tree $\Gamma_p$ with $p$ vertices, $p\geq 2$ can be seen as a tree $\Gamma_{p-1}$ with $p-1$ vertices, to which we add a new vertex on one of its infinite edges, and $n_p-1$ new infinite edges from it. Let us denote by $N$ the number of edges of $\Gamma_{p-1}$. By this transformation $I_N$ becomes an internal edge, parametrized by $[0,a_{p-1}]$, and we have in addition $I_{N+j}$ as external edges, for $1\leq j\leq n_p-1$. We denote $\alpha_{p}$ the  strength of the $\delta$ condition in the new $p^{th}$ vertex.
 The matrix of the new system (unknowns of the $\Gamma_{p-1}$ system, together with an extra-unknown on the new internal line $I_N$, as well as $n_p-1$ unknowns on the new $n_p-1$ external edges) is denoted by $D_{\Gamma_p}(\omega)$. {Notice that if we write the system of unknowns of $\Gamma_p$ by changing the order of the unknowns (i.e. permuting columns) or the order of the conditions at vertices (i.e. permuting lines), then the determinant remains unchanged or it changes sign, and the ratio $\frac{\det \tilde D_{\Gamma_{p}}(\omega)}{\det D_{\Gamma_{p}}(\omega)}$ remains unchanged. }
 
 For $\Gamma_p$, by writing the delta conditions at the end of $I_N$, together with the two  conditions involving the coefficients on $I_N$ at the begining of $I_N$, we obtain the matrix $D_{\Gamma_p}(\omega)$ as
$$\left(\begin{array}{cccc|ccccccccc}
&&&&  & & &&&&&\\
&& {D_{\Gamma_{p-1}}(\omega)}&&& & &&&&&\\
&&&&-1&& & &&&&&\\
&&&&-\frac{\omega}{\omega+\alpha_{p-1}}&& & &&&&&\\ \hline
&&&e^{-\omega a_{p-1}}&e^{\omega a_{p-1}}& -1 & &&&&&\\ 
&&&&&1 & -1  &&&&&\\
&&&&&&.&.&&&&\\&&&&&&&.&.&&\\&&&&&&&&.&.&&\\&&&&&&&&&1&-1&
\\&&&&&&&&&&1&-1\\&&&\frac{-\omega+\alpha_p}{\omega+\alpha_p}e^{-\omega a_{p-1}}&e^{\omega a_{p-1}}&\frac{\omega}{\omega+\alpha_p}&\frac{\omega}{\omega+\alpha_p}&.&.&\frac{\omega}{\omega+\alpha_p}&\frac{\omega}{\omega+\alpha_p}&\frac{\omega}{\omega+\alpha_p}\end{array}\right)$$
and the free term column as
$$T_{\Gamma_p}(\omega)=
\left(\begin{array}{cccccccccccc}
T_{\Gamma_{p-1}}(\omega)\\
\,\\
\frac{t_{N+1}(0,\omega)-t_{N}(a_{p-1},\omega)}{\omega}\\
...\\
\frac{t_{N+n_p-1}(0,\omega)-t_{N+n_p-2}(0,\omega)}{\omega}\\
\,\\
\frac{\omega-\alpha_p}{\omega+\alpha_p}\frac{t_N(a_{p-1},\omega)}{\omega}+\frac{\omega}{\omega+\alpha_p}\sum_{1\leq j\leq n_p-1}\frac{t_{N+j}(0,\omega)}{\omega}\end{array}\right).$$

We point out that $D_{\Gamma_p}$ has $p-1$ pairs of columns that are equals at $\omega=0$. This implies that $\omega=0$ is a zero of order at least $p-1$ for $D_{\Gamma_p}$. The assumption imposed in Theorem \ref{disp-0}  guarantees that  the order of $\omega=0$ is exactly $p-1$.
This will avoid the existence of zero resonances for the resolvent $R_\omega$. In the case when all the strengths $\{\alpha_k\}_{k=1}^{n}$ are positive the condition in Theorem \ref{disp-0} is fulfilled - this will be proved in the first Appendix.

 We shall prove the following Lemma.

\begin{lemma}\label{dets}We have the recursion formulae
$$\det D_{\Gamma_1}(\omega)=\frac{n_1\,\omega+\alpha_1}{\omega+\alpha_1},\,\,\,\frac{\det \tilde D_{\Gamma_{1}}(\omega)}{\det D_{\Gamma_{1}}(\omega)}=\frac{(n_1-2)\,\omega+\alpha_1}{n_1\,\omega+\alpha_1},$$
$$\det D_{\Gamma_p}(\omega)=\frac{n_p\,\omega+\alpha_p}{\omega+\alpha_p}e^{\omega a_{p-1}}\det D_{\Gamma_{p-1}}(\omega)\left(1-\frac{(n_p-2)\,\omega+\alpha_p}{n_p\,\omega+\alpha_p}\,e^{-2\omega a_{p-1}}\frac{\det \tilde D_{\Gamma_{p-1}}(\omega)}{\det D_{\Gamma_{p-1}}(\omega)}\right).$$
$$\frac{\det \tilde D_{\Gamma_{p}}(\omega)}{\det D_{\Gamma_{p}}(\omega)}=\frac{\frac{(n_p-2)\,\omega+\alpha_p}{n_p\,\omega+\alpha_p}-\frac{(n_p-4)\,\omega+\alpha_p}{n_p\,\omega+\alpha_p}\,e^{-2\omega a_{p-1}}\frac{\det \tilde D_{\Gamma_{p-1}}(\omega)}{\det D_{\Gamma_{p-1}}(\omega)}}{1-\frac{(n_p-2)\,\omega+\alpha_p}{n_p\,\omega+\alpha_p}\,e^{-2\omega a_{p-1}}\frac{\det \tilde D_{\Gamma_{p-1}}(\omega)}{\det D_{\Gamma_{p-1}}(\omega)}}.$$
\end{lemma}

\begin{proof}
The part on $\Gamma_1$ was proved in subsection \S\ref{star}.

By developing $\det D_{\Gamma_p}$ with respect to the last $n_p$ lines we obtain an alternated sum of determinants of $n_p\times n_p$ minors composed from the last $n_p$ lines of $D_{\Gamma_p}$ times the determinant of the matrix $D_{\Gamma_p}$ without the lines and columns the minor is made of. On the last $n_p$ lines, there are only $n_p+1$ columns that does  not identically vanish. 
The only possibility to obtain a $n_p\times n_p$ minor composed from the last $n_p$ lines of $D_{\Gamma_p}$ with determinant different from zero is to choose  all last $n_p-1$ columns together with a previous one. This follows from the fact that if we eliminate from $\det D_{\Gamma_n}$  both previous columns together with $n_p-2$ columns among the last $n_p$ columns, we obtain a block-diagonal type matrix, with first diagonal block $D_{\Gamma_{p-1}}$ with its last column replaced by zeros, so its determinant vanishes. Therefore
$$\det D_{\Gamma_p}=\det D_{\Gamma_{p-1}}\det A^{n_p}-\det \tilde D_{\Gamma_{p-1}}\det B^{n_p},$$
where for $m\geq 1$, $A^m$ and $B^m $ are the $m\times m$ matrices 
$$A^{m}=\left(\begin{array}{ccccccccc}e^{\omega a_{p-1}}& -1 & &&&&&\\&1 & -1 & &&&&&\\ &&.&.&&&&\\&&&.&.&&&\\&&&&.&.&&\\&&&&&1&-1&\\&&&&&&1&-1
\\e^{\omega a_{p-1}}&\frac{\omega}{\omega+\alpha_p}&\frac{\omega}{\omega+\alpha_p}&.&.&\frac{\omega}{\omega+\alpha_p}&\frac{\omega}{\omega+\alpha_p}&\frac{\omega}{\omega+\alpha_p} \end{array}\right),$$
$$B^{m}=\left(\begin{array}{ccccccccc}e^{-\omega a_{p-1}}& -1 & &&&&&\\&1 & -1 & &&&&& \\ &&.&.&&&&\\&&&.&.&&&\\&&&&.&.&&\\&&&&&1&-1&\\&&&&&&1&-1
\\\frac{-\omega+\alpha_p}{\omega+\alpha_p}e^{-\omega a_{p-1}}&\frac{\omega}{\omega+\alpha_p}&\frac{\omega}{\omega+\alpha_p}&.&.&\frac{\omega}{\omega+\alpha_p}&\frac{\omega}{\omega+\alpha_p}&\frac{\omega}{\omega+\alpha_p} \end{array}\right).$$
We have 
$$\det A^2=\frac{2\omega+\alpha_p}{\omega+\alpha_p}e^{\omega a_{p-1}},$$ and by developing $A^{m}$ with respect to the first last column we obtain the recursion formula $\det A^m=\frac{\omega}{\omega+\alpha_p} e^{\omega a_{p-1}}+\det A^{n_p-1},$ so
$$\det A^m=\frac{m\,\omega+\alpha_p}{\omega+\alpha_p}e^{\omega a_{p-1}}.$$
Similarly we obtain 
$$\det B^m=\frac{(m-2)\,\omega+\alpha_p}{\omega+\alpha_p}e^{-\omega a_{p-1}}.$$
Therefore we find indeed
$$\det D_{\Gamma_p}(\omega)=\frac{n_p\,\omega+\alpha_p}{\omega+\alpha_p}e^{\omega a_{p-1}}\det D_{\Gamma_{p-1}}(\omega)\left(1-\frac{(n_p-2)\,\omega+\alpha_p}{n_p\,\omega+\alpha_p}\,e^{-2\omega a_{p-1}}\frac{\det \tilde D_{\Gamma_{p-1}}(\omega)}{\det D_{\Gamma_{p-1}}(\omega)}\right).$$

In a similar way we get
$$\det \tilde D_{\Gamma_p}(\omega)=\frac{(n_p-2)\,\omega+\alpha_p}{\omega+\alpha_p}e^{\omega a_{p-1}}\det D_{\Gamma_{p-1}}(\omega)-\frac{(n_p-4)\,\omega+\alpha_p}{\omega+\alpha_p}e^{-\omega a_{p-1}}\det \tilde D_{\Gamma_{p-1}}(\omega),$$
so
$$\frac{\det \tilde D_{\Gamma_{p}}(\omega)}{\det D_{\Gamma_{p}}(\omega)}=\frac{\frac{(n_p-2)\,\omega+\alpha_p}{n_p\,\omega+\alpha_p}-\frac{(n_p-4)\,\omega+\alpha_p}{n_p\,\omega+\alpha_p}\,e^{-2\omega a_{p-1}}\frac{\det \tilde D_{\Gamma_{p-1}}(\omega)}{\det D_{\Gamma_{p-1}}(\omega)}}{1-\frac{(n_p-2)\,\omega+\alpha_p}{n_p\,\omega+\alpha_p}\,e^{-2\omega a_{p-1}}\frac{\det \tilde D_{\Gamma_{p-1}}(\omega)}{\det D_{\Gamma_{p-1}}(\omega)}},$$
and the proof of the Lemma is complete.
\end{proof}

\subsection{A lower bound for $\det D_{\Gamma_p}(i\tau)$ away from $0$}

\begin{lemma}\label{infboundaway}Function $\det D_{\Gamma_p}(\omega)$ is lower bounded by a positive constant on a strip containing the imaginary axis, away from zero: 
$$\forall \delta>0,\,\exists c_{\Gamma_p},\epsilon_{\Gamma_p}>0,\exists 0<r_{\Gamma_p}<1\, \text{s.t.}\,|\det D_{\Gamma_p}(\omega)|>c_{\Gamma_p},\,\left|\frac{\det \tilde D_{\Gamma_p}(\omega)}{\det D_{\Gamma_p}(\omega)}\right|<r_{\Gamma},$$
for all $\omega\in\mathbb C$ with $|\Re\omega|<\epsilon_{\Gamma_p}$ and $ |\Im\omega|>\delta.$
\end{lemma}

\begin{proof}
We shall prove this Lemma by recursion on $p$. For $p=1$ Lemma \ref{dets} insures us that
$$\det D_{\Gamma_1}(\omega)=\frac{n_1\,\omega+\alpha_1}{\omega+\alpha_1},\quad \frac{\det \tilde D_{\Gamma_{1}}(\omega)}{\det D_{\Gamma_{1}}(\omega)}=\frac{(n_1-2)\,\omega+\alpha_1}{n_1\,\omega+\alpha_1}.$$ 
We obtain a positive lower bound for $|\det D_{\Gamma_1}(\omega)|$ if we avoid that it approaches zero. Therefore the existence of $c_{\Gamma_1}>0$ is obtained by considering $\epsilon_{\Gamma_1}\leq \frac{|\alpha_1|}{2n_1}$. 
Next, we have
$$\Big|\frac{(n_1-2)\,\omega+\alpha_1}{n_1\,\omega+\alpha_1}\Big|<1\iff 0<\alpha_1\Re\omega+(n_1-1)|\omega|^2,$$
so for any $\delta>0$ we get an appropriate $0<r_{\Gamma_1}<1$ by choosing
$$\epsilon_{\Gamma_1}\leq \frac{(n_1-1)\delta^2}{2|\alpha_1|}.$$

Assume that we have proved this Lemma for $p-1$. We shall show now that it also holds  for $p$.
Now, from ratio information part in this Lemma for $\Gamma_{p-1}$ we can choose $\epsilon_{\Gamma_p}$ small enough to have for $|\Re \omega|<\epsilon_{\Gamma_p}$ and $|\Im\omega|>\delta$ 
$$\Big |1-\frac{(n_p-2)\,\omega+\alpha_p}{n_p\,\omega+\alpha_p}\,e^{-2\omega a_{p-1}}\frac{\det \tilde D_{\Gamma_{p-1}}(\omega)}{\det D_{\Gamma_{p-1}}(\omega)}\Big|>c_0>0.$$
Also from this Lemma for $\Gamma_{p-1}$ we have the existence of two positive constants  $ c_{\Gamma_{p-1}}$ and $\epsilon_{\Gamma_{p-1}}$ such that $|\det D_{\Gamma_{p-1}}(\omega)|>c_{\Gamma_{p-1}},\,\forall\omega\in\mathbb C, |\Re\omega|<\epsilon_{\Gamma_{p-1}}$ and $|\Im\omega|>\delta$. Finally, $\frac{n_p\omega+\alpha_p}{\omega+\alpha_p}$ is lower bounded by a positive constant for $\Re\omega$ small enough, so eventually we get 
$$\exists c_{\Gamma_{p}},\epsilon_{\Gamma_{p}}>0,\,|\det D_{\Gamma_{p}}(\omega)|>c_{\Gamma_{p}},\,\forall\omega\in\mathbb C, |\Re\omega|<\epsilon_{\Gamma_{p}},|\Im\omega|>\delta.$$ 
We are left with showing that the ratio $\frac{\det \tilde D_{\Gamma_p}(\omega)}{\det D_{\Gamma_p}(\omega)}$ is of modulus less than one. In view of the recursion formula on the ratio from Lemma \ref{dets}, {we first impose as a condition on $\epsilon_{\Gamma_p}$ that $$\tilde r_{\Gamma_{p-1}}:=e^{2\epsilon_{\Gamma_p}a_{p-1}}r_{\Gamma_{p-1}}<1,$$
and then we have to show that for $|z|<\tilde r_{\Gamma_{p-1}}$
$$\left|\frac{(n_p-2)\,\omega+\alpha_p-[(n_p-4)\,\omega+\alpha_p]z}{{n_p\,\omega+\alpha_p}-[(n_p-2)\,\omega+\alpha_p]z}\right|<r_{\Gamma_p},$$
for all complex $\omega$ with $|\Re\omega|<\epsilon_{\Gamma_p}$ and $ |\Im\omega|>\delta,$ for $\epsilon_{\Gamma_p}$ to be chosen and $r_{\Gamma_p}<1$}. Denoting $q=(n_p-2)\omega+\alpha_p$, the above inequality is written in as
$$ |q- (q-2\omega)z|<|(q+2\omega)-qz|\iff |q(1-z)+2\omega z|<|q(1-z)+2\omega|.$$
Expanding this last inequality we find that we have to prove that
$$0< |\omega|^2 (1-|z|^2)+|1-z|^2\Big(  (n_p-2)|\omega|^2+\alpha_p \Re (\omega) \Big).$$
Since $n_p\geq 2$ and $|z|<\tilde r_{\Gamma_{p-1}}<1$, it is enough to have
$$0< |\omega|^2 (1-|z|^2)+|1-z|^2\alpha_p \Re (\omega).$$
Also, $|\Re z|<\tilde r_{\Gamma_{p-1}}<1$, so by choosing 
$$\epsilon_{\Gamma_p}\leq \frac{(1-\tilde r_{\Gamma_{p-1}}^2)\delta^2}{2|\alpha_p|(1-\tilde r_{\Gamma_{p-1}})^2},$$ 
we get the existence of $r_{\Gamma_p}<1$. 
\end{proof}

\subsection{Vanishing of the numerator at $\tau=0$}Recall that we have denoted by $M^{c_e}_{\Gamma_p}(\omega)$ (respectively $\det M^{\tilde c_e}_{\Gamma_p}(\omega)$) the matrix $D_{\Gamma_p}(\omega)$ with the column corresponding to the unknown  $c_e$ (respectively $\tilde c_e$), $D_{\Gamma_p}^{e}(\omega)$  (respectively $D_{\Gamma_p}^{\tilde e}(\omega)$),  replaced by the free terms column $T_{\Gamma_p}(\omega)$.
In particular $\omega \det M^{c_e}_{\Gamma_p}(\omega)$ (respectively $\omega \det M^{\tilde c_e}_{\Gamma_p}(\omega)$) is the determinant of the matrix $D_{\Gamma_p}(\omega)$ with the column corresponding to the unknown  $c_e$ (respectively $\tilde c_e$) replaced by $\omega T_{\Gamma_p}(\omega)$.

\begin{lemma}\label{identity-lemma}The following holds
\begin{equation}\label{identity}
-(\omega T_{\Gamma_p}(\omega))(0)=\sum _{e\in \mathcal{E}} t_e(0,0)D_{\Gamma_p}^{e}(0)+\sum _{e\in \mathcal{I}} t_e(0,0)D_{\Gamma_p}^{\tilde e}(0)
\end{equation}
\end{lemma}

\begin{rem}\label{rem-identity}
From the shape of $D_{\Gamma_p}(\omega)$ displayed in the proof of Lemma \ref{dets} we notice that the two junction columns with  $D_{\Gamma_{p-1}}(\omega)$, corresponding to the coefficients of the resolvent on the connecting edge $I_N$, are
$$D_{\Gamma_p}^{I_N}(\omega)=^t\left(0,\dots,0,-1,\frac{\omega}{\omega+\alpha_{p-1}},e^{-\omega a_{p-1}},0,\dots,0,\frac{-\omega+\alpha_p}{\omega+\alpha_p}e^{-\omega a_{p-1}}\right)$$
and 
$$D_{\Gamma_p}^{\tilde I_N}(\omega)=^t\left(0,\dots,0,-1,-\frac{\omega}{\omega+\alpha_{p-1}},e^{\omega a_{p-1}},0,\dots,0,e^{\omega a_{p-1}}\right).$$
In particular, these two columns are the same at $\omega=0$. Moreover, $D_{\Gamma_p}(\omega)$ contains $p-1$ such pair of columns, 
$D_{\Gamma_p}^{e}(0)=D_{\Gamma_p}^{\tilde e}(0)$ for all $e\in \mathcal {I}$. Thus, the last term in the right hand side of  \eqref{identity} could be  $D_{\Gamma_p}^{ e}(0)$ either $D_{\Gamma_p}^{\tilde e}(0)$, $e\in \mathcal{I}$.
\end{rem}

\begin{proof}We will prove this identity inductively. In the case $p=1$ we use that $(\omega T_{\Gamma_1})$ is given in Section \ref{star}.
We choose $X_1=(t_1(0,0),t_2(0,0), \dots, t_{n_1}(0,0))$ and then $D_{\Gamma_1}(0)X_1=-(\omega T_{\Gamma_1})(0)$ which proves \eqref{identity} when $p=1$.
 
 Given now $X_{p-1}$ such that $D_{\Gamma_{p-1}}(0)X_{p-1}=-(\omega T_{\Gamma_{p-1}}(\omega))(0)$ we construct 
 $X_p$ as follows 
 $$^tX_p=(^tX_{p-1}, 0, t_{N+1}(0,0), \dots, t_{N+n_p-1}(0,0)).$$
 Using the recursion between $D_{\Gamma_{p}}$ and $D_{\Gamma_{p-1}}$ used in the proof of Lemma \ref{dets}, identity
$$\omega T_{\Gamma_p}(\omega)=\left(\begin{array}{ccccccccccc}\omega T_{\Gamma_{p-1}}(\omega)\\
t_{N+1}(0,\omega)-t_{N}(a_{p-1},\omega)\\
...\\
t_{N+n_p-1}(0,\omega)-t_{N+n_p-2}(0,\omega)\\
\frac{\omega-\alpha_p}{\omega+\alpha_p}t_N(a_{p-1},\omega)+\frac{\omega}{\omega+\alpha_p}\sum_{1\leq j\leq n_p-1}t_{N+j}(0,\omega)\end{array}\right),$$
and the fact that $t_e(0,0)=t_e(a_e,0)$ for all $e\in \mathcal{I}$,
  we obtain that $X_p$ satisfies the system
$D_{\Gamma_{p}}(0)X_{p}=-(\omega T_{\Gamma_{p}}(\omega))(0)$. Writing  this identity in terms of the columns of matrix $D_{\Gamma_{p}}(0)$ we obtain the desired identity.
  \end{proof}

\begin{lemma}\label{numerator}
$\omega=0$ is a root of order at least $p-1$ of $ \omega\det M^{c_e}_{\Gamma_p}(\omega)$ and of  $ \omega\det M^{\tilde c_e}_{\Gamma_p}(\omega)$ for all edge $e$.
\end{lemma}

\begin{proof}
We shall perform the proof for $ \omega\det M^{c_e}_{\Gamma_p}(\omega)$; the result for $ \omega\det M^{\tilde c_e}_{\Gamma_p}(\omega)$ will be the same.
From the shape of $D_{\Gamma_p}(\omega)$ displayed in the proof of Lemma \ref{dets} and Remark \ref{rem-identity} we have $p-1$ pairs of columns that are equal at $\omega=0$.
Moreover, by Lemma \ref{identity-lemma}, $(\omega T_{\Gamma_p})(0)$ is a linear combination of these columns evaluated at $\omega=0$.

 The derivative of a determinant is the sum  of the determinants of the matrices obtained by differentiating one column.
 When $T_{\Gamma_p}$ does not replace any of these $2(p-1)$ columns 
  it follows  that the result of this Lemma holds since there are always two columns identically.
     Then {by the above argument we have already}
\begin{equation}\label{p-2}
\partial_\omega^k (\omega \det M^{c_e}_{\Gamma_p})(0)=0,\,\,\forall \,0\leq k\leq p-3.
\end{equation}

   Assume now that $T_{\Gamma_p}$ replaces one of these $2(p-1)$ columns. 
For proving the Lemma we are left to show that
$$\partial_\omega^{p-2} (\omega\det M^{c_e}_{\Gamma_p})(0)=0.$$

Using again the fact that $D_{\Gamma_p}(\omega)$ contains $p-1$ pairs of columns that are the same two by two at $\omega=0$, we only need to show that $\det A_{\Gamma_p}(0)=0$, where $A_{\Gamma_p}(\omega)$ is $D_{\Gamma_p}(\omega)$ with the column $\omega T_{\Gamma_p}(\omega)$ replacing one column of one pair, and one column of each remaining $p-2$ pairs of columns is differentiated. In particular $A_{\Gamma_p}(0)$ contains one column of each $p-1$ pairs unchanged.
Since by Lemma \ref{identity-lemma} we know that $(\omega T_{\Gamma_p}(\omega))(0)$
is a linear combination of the columns corresponding to external edges and of the internal ones (each one from the $p-1$ pairs) the new determinant vanishes and the proof is finished.
\end{proof}

\begin{lemma}\label{coefs}
For all edges indices $\lambda$ and $e$, $\omega=0$ is a root of order at least $p-2$ for the coefficient $f_{\lambda,e}(\omega)$ of $t_\lambda(0,\omega)$ in $\omega \det M^{c_e}_{\Gamma_p}(\omega)$, and the same holds for the coefficient $\tilde f_{\lambda,e}(\omega)$ of $t_\lambda(0,\omega)$ in $\omega \det M^{\tilde c_e}_{\Gamma_p}(\omega)$.\end{lemma}
\begin{proof}
This result follows from the discussion that led to \eqref{p-2}: the matrix $\omega M^{c_e}_{\Gamma_p}(\omega)$ has $p-2$ pairs of columns that are identical at $\omega=0$.
\end{proof}
\begin{lemma}\label{coefsext}
For all edge index $e$ and all external edge index $\lambda$, $\omega=0$ is a root of order at least $p-1$ for the coefficient $f_{\lambda,e}(\omega)$ of $t_\lambda(0,\omega)$ in $\omega \det M^{c_e}_{\Gamma_p}(\omega)$, and the same holds for the coefficient $\tilde f_{\lambda,e}(\omega)$ of $t_\lambda(0,\omega)$ in $\omega \det M^{\tilde c_e}_{\Gamma_p}(\omega)$.
\end{lemma}
\begin{proof}

%
%
%

The statement corresponds to the particular case of Lemma \ref{numerator} where all the components of  $T_{\Gamma_p}$ are taken to be zero except $t_{\lambda}(0,\omega)$  which is replaced by one.
\end{proof}

\begin{lemma}\label{coefsint}
For all edge index $e$ and all internal edge index $\lambda$, $\omega=0$ is a root of order at least $p-1$ for $f^1_{\lambda,e}(\omega)+f^2_{\lambda,e}(\omega)$ where $f^1_{\lambda,e}(\omega)$ is the coefficient of $t_\lambda(0,\omega)$ in $\omega \det M^{c_e}_{\Gamma_p}(\omega)$ and $f^2_{\lambda,e}(\omega)$ is the coefficient of $t_\lambda(a_\lambda,\omega)$ in $\omega \det M^{c_e}_{\Gamma_p}(\omega)$. Also,  the same holds for $\tilde f^1_{\lambda,e}(\omega)+\tilde f^2_{\lambda,e}(\omega)$, where $\tilde f^1_{\lambda,e}(\omega)$ is the coefficient of $t_\lambda(0,\omega)$ in $\omega\det M^{\tilde c_e}_{\Gamma_p}(\omega)$ and $\tilde f^2_{\lambda,e}(\omega)$ is the coefficient of $t_\lambda(a_\lambda,\omega)$ in $\omega\det M^{\tilde c_e}_{\Gamma_p}(\omega)$
\end{lemma}
\begin{proof}
The proof goes the same as for Lemma \ref{coefsext}.
\end{proof}

\subsection{The end of the proof}\label{sect:end}
Now we shall use the theorem hypothesis, $\partial _{\omega}^{(p-1)}\det D_{\Gamma_p}\,_{\scriptscriptstyle{\vert \omega=0}}\neq 0$.  We obtain that $\omega=0$ is a root of order  $p-1$ of $\det D_{\Gamma_p}$.
From the previous subsections we conclude the following result.
\begin{lemma}\label{ancontres} Function $\omega R_\omega \bold f(x)$ can be analytically continued in a region containing the imaginary axis.
\end{lemma}
\begin{proof}The proof is an immediate consequence of decomposition \eqref{ressum}:
\begin{equation}\label{reso}
R_\omega \bold u_0(x)=\frac{\det M^{c_e}_{\Gamma_p}(\omega)}{\det D_{\Gamma_p}(\omega)}\,e^{\omega x}+\frac{\det M^{\tilde c_e}_{\Gamma_p}(\omega)}{\det D_{\Gamma_p}(\omega)}\,e^{-\omega x}+\frac{t_e(x,\omega)}{\omega},
\end{equation}
 for $x\in I_e$, and Lemma \ref{infboundaway},  Lemma \ref{numerator} and the fact that $\omega=0$ is a root of order $p-1$ for $\det D_{\Gamma_p}$.
\end{proof}

\begin{proof}[Proof of Theorem \ref{disp}]
As a consequence of Lemma \ref{ancontres} we can use a spectral calculus argument to write the solution of the Schr\"odinger equation with initial data $\bu_0$ as
\begin{equation}\label{SL}
e^{-it H_\alpha^\Gamma }P\bu_0(x)=\frac{1}{i\pi}
\int_{-\infty}^{\infty}e^{-it\tau^2}\tau R_{i\tau}\bu_0(x){d\tau}.
\end{equation}
{In view of the definition of $t_e$ and with the notations from Lemma \ref{coefsext} and Lemma  \ref{coefsint} we can also write the decomposition \eqref{reso} as
\begin{equation}\label{decomposition}
\tau R_{i\tau} \bu_0(x)=\frac{1}{2}\int_{I_e}\bu_0\,e^{-i\tau|x-y|}dy+\sum_{\lambda \in \mathcal{E}}\frac{f_{\lambda,e}(i\tau)}{\det D_{\Gamma_p}(i\tau)}\int_{I_\lambda}\bu_0(y)
e^{i\tau y}dy\,e^{i\tau x}
\end{equation}
$$+\sum_{\lambda\in \mathcal{E}}\frac{\tilde f_{\lambda,e}(i\tau)}{\det D_{\Gamma_p}(i\tau)}\int_{I_\lambda}\bu_0(y)
e^{i\tau y}dy\,e^{-i\tau x}$$
$$+\sum_{\lambda \in \mathcal{I} }\int_{I_\lambda}\bu_0(y)
\left(e^{i\tau y}\frac{f^1_{\lambda,e}(i\tau)}{\det D_{\Gamma_p}(i\tau)}+e^{i\tau(a_\lambda-y)}\frac{f^2_{\lambda,e}(i\tau)}{\det D_{\Gamma_p}(i\tau)}\right)dy\,e^{i\tau x}$$
$$+\sum_{\lambda \in \mathcal{I}}\int_{I_\lambda}\bu_0(y)
\left(e^{i\tau y}\frac{\tilde f^1_{\lambda,e}(i\tau)}{\det D_{\Gamma_p}(i\tau)}+e^{i\tau(a_\lambda-y)}\frac{\tilde f^2_{\lambda,e}(i\tau)}{\det D_{\Gamma_p}(i\tau)}\right)dy\,e^{-i\tau x}.$$
Moreover, in view of the results in Lemma \ref{coefsint} and Lemma \ref{coefsext} we gather the terms as follows
\begin{equation}\label{decompositionbis}
\tau R_{i\tau} \bu_0(x)=\frac{1}{2}\int_{I_e}\bu_0\,e^{-i\tau|x-y|}dy
\end{equation}
$$+\sum_{\lambda \in \mathcal{E}}\int_{I_\lambda}\bu_0(y)\,\frac{f_{\lambda,e}(i\tau)}{\det D_{\Gamma_p}(i\tau)}
e^{i\tau (x+y)}\,dy+\sum_{\lambda\in \mathcal{E}}\int_{I_\lambda}\bu_0(y)\,\frac{\tilde f_{\lambda,e}(i\tau)}{\det D_{\Gamma_p}(i\tau)}
e^{i\tau (y-x)}\,dy$$
$$+\sum_{\lambda \in \mathcal{I}}\int_{I_\lambda}\bu_0(y)
\,\frac{f^1_{\lambda,e}(i\tau)+f^2_{\lambda,e}(i\tau)}{\det D_{\Gamma_p}(i\tau)}\,e^{i\tau (x+y)}\,dy+\sum_{\lambda \in \mathcal{I}}\int_{I_\lambda}\bu_0(y)
\,\frac{\tilde f^1_{\lambda,e}(i\tau)+\tilde f^2_{\lambda,e}(i\tau)}{\det D_{\Gamma_p}(i\tau)}\,e^{i\tau (y-x)}\,dy$$
$$+\sum_{\lambda \in \mathcal{I} }\int_{I_\lambda}\bu_0(y)\,
\frac{\left(e^{i\tau(a_\lambda-y)}-e^{i\tau y}\right)f^2_{\lambda,e}(i\tau)}{\det D_{\Gamma_p}(i\tau)}\,e^{i\tau x}\,dy$$
$$+\sum_{\lambda \in \mathcal{I}}\int_{I_\lambda}\bu_0(y)\,
\frac{\left(e^{i\tau(a_\lambda-y)}-e^{i\tau y}\right)\tilde f^2_{\lambda,e}(i\tau)}{\det D_{\Gamma_p}(i\tau)}\,e^{-i\tau x}\,dy.$$}

Let $e$ be an external edge. In view of  Lemma \ref{coefsext} and the fact that $\omega=0$ is a root of order  $p-1$ of $\det D_{\Gamma_p}$, we obtain that the fraction $\frac{f_{\lambda,e}(i\tau)}{\det D_{\Gamma_p}(i\tau)}$ is upper bounded near $\tau=0$. Outside a neighbourhood of $\tau=0$ we use   Lemma \ref{infboundaway} to infer that  
$|\det D_{\Gamma_p}(i\tau)|$
  is positively lower bounded outside neighbourhoods of $\tau=0$. Moreover, in view of the explicit entries of $M^{c_e}_{\Gamma_p}(i\tau)$, we see that $f_{\lambda,e}(i\tau)$ is upper bounded for any $\tau\in \R$ since all the entries of matrix $D_{\Gamma_p}(i\tau)$ as well as the coefficients of $t_\lambda$ in $T_{\Gamma_p(i\tau)}$ have absolute value less than one.  Summarizing, we have obtained that 
$$\frac{f_{\lambda,e}(i\tau)}{\det D_{\Gamma_p}(i\tau)}\in L^\infty(\R).$$ 
The derivative of this fraction is upper-bounded near $\tau=0$ by limited development at $\tau=0$. Outside neighbourhoods of $\tau=0$ we have that  $\partial_\tau f_{\lambda,e}(i\tau)$ and $\partial_\tau\det D_{\Gamma_p}(i\tau)$  have upper bounds of type $\frac{1}{\tau^2}$. This is because in each term of $\partial_\tau f_{\lambda,e}(i\tau)$ and $\partial_\tau\det D_{\Gamma_p}(i\tau)$ contains a derivative of an element of the line given by the $\delta$-condition involving the derivatives in the root vertex $\mathcal{O}$. This vertex is the one which is an initial vertex for all {$n$} edges emerging from it: $I(e)=\mathcal{O},\forall e\in E,\mathcal{O}\in e$. If $\alpha$ denotes the strength of the $\delta$-condition in $\mathcal O$, then this line of the matrix $D_{\Gamma_p}(i\tau)$ is composed by $0,1$ and $\pm\frac{i\tau}{i\tau+\alpha}$, 
where the minus sign appears only on  the finite edges that stars from $\mathcal{O}$,
and this line for {the column matrix} $i \tau T_{\Gamma_p}(i\tau)$ is 
$$\left(\frac{i\tau -\alpha}{i\tau+\alpha}\,t_1(0,i \tau)+\frac{i\tau}{i\tau+\alpha}\sum_{2\leq j\leq n}t_j(0,i\tau )\right).$$
Finally, as above, $f_{\lambda,e}(i\tau)$ and $\det D_{\Gamma_p}(i\tau)$ are upper bounded and from Lemma \ref{infboundaway} we have that $|\det D_{\Gamma_p}(i\tau)|$  is positively lower bounded outside neighbourhoods of $\tau=0$. 
As a conclusion we infer that 
$$\partial_\tau \frac{f_{\lambda,e}(i\tau)}{\det D_{\Gamma_p}(i\tau)}\in L^1(\mathbb R).$$
The same argument using Lemma \ref{coefsext}, Lemma \ref{coefsint} and Lemma \ref{coefs} can be performed to obtain that 
$$\frac{\tilde f_{\lambda,e}(i\tau)}{\det D_{\Gamma_p}(i\tau)},\quad \frac{f^1_{\lambda,e}(i\tau)+f^2_{\lambda,e}(i\tau)}{\det D_{\Gamma_p}(i\tau)},\quad \frac{\tilde f^1_{\lambda,e}(i\tau)+\tilde f^2_{\lambda,e}(i\tau)}{\det D_{\Gamma_p}(i\tau)},$$
$$\frac{\left(e^{i\tau(a_\lambda-y)}-e^{i\tau y}\right)f^2_{\lambda,e}(i\tau)}{\det D_{\Gamma_p}(i\tau)},\quad \frac{\left(e^{i\tau(a_\lambda-y)}-e^{i\tau y}\right)\tilde f^2_{\lambda,e}(i\tau)}{\det D_{\Gamma_p}(i\tau)}$$
are in $L^\infty$ with derivative in $L^1$. Notice that when $\lambda$ belongs to an internal edge $I_\lambda$ it follows that the interval $I_\lambda$ have finite length.  Therefore for the last fractions we use  that $\left(e^{i\tau(a_\lambda-y)}-e^{i\tau y}\right) f^2_{\lambda,e}(i\tau)$ vanishes or order $p-1$ at $\tau=0$ and repeat the argument used above. The only difference with the previous cases is that we will obtain bounds in terms of parameter $y$. Since $y$ is now on an internal edge  $I_\lambda$ of finite length we obtain uniform bounds.
Therefore the dispersion estimate \eqref{dispersion} of Theorem \ref{disp} follows from \eqref{SL} by using \eqref{decompositionbis} and the classical oscillatory integral estimate
\begin{equation}\label{IBP}\left|\int_{-\infty}^{\infty}e^{-it\tau^2}e^{i\tau a} g(\tau)d\tau\right|\leq \frac{C}{\sqrt{|t|}}\left(\|g\|_{L^\infty}+\|g'\|_{L^1}\right).\end{equation}

\end{proof}

\section{Appendix: The multiplicity of the root $\omega=0$ of $\det D_{\Gamma_p}(\omega)$}
In this Appendix we prove that the condition \eqref{cond} is fullfilled in the case of positive strengths. We shall show first the following double property. 
\begin{lemma}\label{ratio} For all $p\geq 1$ we have the following informations
$$(\mathcal P_p^1):\,\frac{\det \tilde D_{\Gamma_{p}}}{\det D_{\Gamma_{p}}}(0)=1\,\,\,,\,\,\,(\mathcal P_p^2):\,\partial_\omega\left(\frac{\det \tilde D_{\Gamma_{p}}}{\det D_{\Gamma_{p}}}\right)(0)<0.$$
\end{lemma}
\begin{proof}Lemma \ref{dets} insures us that $\frac{\det \tilde D_{\Gamma_{1}}}{\det D_{\Gamma_{1}}}(\omega)=\frac{(n_1-2)\,\omega+\alpha_1}{n_1\,\omega+\alpha_1}$, and in particular
$$\partial_\omega\left(\frac{\det \tilde D_{\Gamma_{1}}}{\det D_{\Gamma_{1}}}\right)(\omega)=-\frac{2\alpha_1}{(n_1\omega+\alpha_1)^2},$$
and the Lemma follows for $p=1$, {since $\alpha_1>0$}. We shall show the general case by recursion. Let us denote by $P_p(\omega)$ and $Q_p(\omega)$ the numerator and respectively the denominator in recursion formula of the ratio from Lemma \ref{dets}
$$P_p(\omega)=\frac{(n_p-2)\,\omega+\alpha_p}{n_p\,\omega+\alpha_p}-\frac{(n_p-4)\,\omega+\alpha_p}{n_p\,\omega+\alpha_p}\,e^{-2\omega a_{p-1}}\frac{\det \tilde D_{\Gamma_{p-1}}(\omega)}{\det D_{\Gamma_{p-1}}(\omega)},$$ 
$$Q_p(\omega)=1-\frac{(n_p-2)\,\omega+\alpha_p}{n_p\,\omega+\alpha_p}\,e^{-2\omega a_{p-1}}\frac{\det \tilde D_{\Gamma_{p-1}}(\omega)}{\det D_{\Gamma_{p-1}}(\omega)}.$$
We have $P_p(0)=Q_p(0)=0$, and in view of $(\mathcal P_{p-1}^1)$ we compute
$$\partial_\omega P_p(0)=\partial_\omega Q_p(0)=\frac 2\alpha_p+2a_{p-1}-\partial_\omega\left(\frac{\det \tilde D_{\Gamma_{p-1}}}{\det D_{\Gamma_{p-1}}}\right)(0).$$
Therefore $(\mathcal P_{p-1}^2)$ insures us that $\partial_\omega P_p(0)=\partial_\omega Q_p(0)\neq 0$ and we apply l'H\^opital's rule to conclude $(\mathcal P_{p}^1)$.\\

Since
$$P_p(\omega)-Q_p(\omega)=-\frac{2\omega}{n_p\omega+\alpha_p}\left(1-e^{-2a_{p-1}\omega}\frac{\det \tilde D_{\Gamma_{p-1}}}{\det D_{\Gamma_{p-1}}}(\omega)\right).$$
we define $\tilde P_p(\omega)$ and $\tilde Q_p(\omega)$ by
$$P_p(\omega)=\frac{2\omega}{n_p\omega+\alpha_p}\tilde P_p(\omega),\,\,Q_p(\omega)=\frac{2\omega}{n_p\omega+\alpha_p}\tilde Q_p(\omega).$$
In particular
$$\frac{\det \tilde D_{\Gamma_{p}}}{\det D_{\Gamma_{p}}}(\omega)=\frac{\tilde P_p}{\tilde Q_p}(\omega), \quad\tilde P_p(\omega)-\tilde Q_p(\omega)=-\left(1-e^{-2a_{p-1}\omega}\frac{\det \tilde D_{\Gamma_{p-1}}}{\det D_{\Gamma_{p-1}}}(\omega)\right).$$
By using $(\mathcal P_{p-1}^1)$ and $(\mathcal P_{p-1}^2)$
$$\partial_\omega(\tilde P_p-\tilde Q_p)(0)=-2a_{p-1}+\partial_\omega\left(\frac{\det \tilde D_{\Gamma_{p-1}}}{\det D_{\Gamma_{p-1}}}\right)(0)$$
Moreover,
 $$\tilde P_p(0)=\tilde Q_p(0)=\frac{\alpha_p}{2}\,\partial_\omega P_p(0)=\frac{\alpha_p}{2}\,\partial_\omega Q_p(0)\neq 0,$$
and we can compute
$$\partial_\omega\left(\frac{\det \tilde D_{\Gamma_{p}}}{\det D_{\Gamma_{p}}}\right)(0)=\frac{\partial_\omega\tilde P_p(0)\tilde Q_p(0)-\tilde P_p(0)\partial_\omega\tilde Q_p(0)}{(\tilde Q_p(0))^2}=\frac{\partial_\omega(\tilde P_p-\tilde Q_p)(0)}{\tilde Q_p(0)}$$
$$=-\frac{2a_{p-1}-\partial_\omega\left(\frac{\det \tilde D_{\Gamma_{p-1}}}{\det D_{\Gamma_{p-1}}}\right)(0)}
{\frac{\alpha_p}{2}\left(\frac 2{\alpha_p}+2a_{p-1}-\partial_\omega\left(\frac{\det \tilde D_{\Gamma_{p-1}}}{\det D_{\Gamma_{p-1}}}\right)(0)\right)}.$$
By using again $(\mathcal P_{p-1}^2)$ we obtain $(\mathcal P_{p}^2)$.

\end{proof}

\begin{lemma}\label{root}
$\omega=0$ is a root of order $p-1$ of $\det D_{\Gamma_p}(\omega)$. In particular, condition \eqref{cond} is fulfilled.
\end{lemma}

\begin{proof}From Lemma \ref{dets} we have $\det D_{\Gamma_1}(\omega)=n_1\,\omega+\alpha$, so $\det D_{\Gamma_1}(0)\neq 0$. Lemma \ref{dets} provides us the expression
$$\det D_{\Gamma_p}(\omega)=\frac{n_p\,\omega+\alpha_p}{\omega+\alpha_p}e^{\omega a_{p-1}}\det D_{\Gamma_{p-1}}(\omega)\left(1-\frac{(n_p-2)\,\omega+\alpha_p}{n_p\,\omega+\alpha_p}\,e^{-2\omega a_{p-1}}\frac{\det \tilde D_{\Gamma_{p-1}}}{\det D_{\Gamma_{p-1}}}(\omega)\right),$$
so by recursion it is enough to show that $\omega=0$ is a simple root for
$$1-\frac{(n_p-2)\,\omega+\alpha_p}{n_p\,\omega+\alpha_p}\,e^{-2\omega a_{p-1}}\frac{\det \tilde D_{\Gamma_{p-1}}}{\det D_{\Gamma_{p-1}}}(\omega).$$
This expression is precisely $Q_p(\omega)$ from the proof of Lemma \ref{ratio}, and it was proved there that $\partial_\omega Q_p(0)\neq 0$.
\end{proof}

\section{Appendix: Strichartz estimates}\label{appStrich}

In this Appendix we prove Theorem \ref{strichartz}. Let us first remark that by the definition of $P_e$ we have
$$P_e{\bf{\phi}}=\sum _{k=1}^m <\phi,\varphi_k> \varphi_k,$$
where $\{\varphi_k\}_{k=1}^m$ are eigenfunctions of operator $H$. Since $\varphi_k\in L^2(\Gamma)$  we have that $\varphi_k\in L^1(\Gamma)\cap L^\infty(\Gamma)$. Indeed, on the infinite edges the eigenfunctions corresponding to an eigenvalue $\lambda<0$ are of type $C\exp(-\sqrt{-\lambda})x$. This means that they belong to $L^1(e)\cap L^\infty(e)$ for any external edge $e$. On the 
internal edges this property trivially holds. 

Then $P_e$ is defined for any $\phi\in L^r(\Gamma)$, $1\leq r\leq \infty$, and
for any $1\leq r_1,r_2\leq \infty$ we have by H\"older inequality:
\begin{align*}
\|P_e\phi\|_{L^{r_2}(\Gamma)}&\leq \sum _{k=1}^m| <\phi,\varphi_k>| \|\varphi_k\|_{L^{r_2}(\Gamma)}\\
&\leq \|\phi\|_{L^{r_1}(\Gamma)}\sum _{k=1}^m \|\varphi_k\|_{L^{r_1'}(\Gamma)} \|\varphi_k\|_{L^{r_2}(\Gamma)}\leq C(\Gamma,r_1,r_2) \|\phi\|_{L^{r_1}(\Gamma)}.
\end{align*}

\begin{proof}[Proof of Theorem \ref{strichartz}]
Using the dispersive estimate \eqref{dispersion} and the mass conservation
$$\|e^{-itH}\bu_0\|_{L^2(\Gamma)}=\|\bu_0\|_{L^2(\Gamma)}$$
we obtain by applying the classical $TT^*$ argument and Christ-Kiselev Lemma \cite{0974.47025} the following estimates 
\begin{equation}\label{est.1}
\|e^{-itH}P\bu_0\|_{L^q(\R,L^r(\Gamma))}\leq C\|\bu_0\|_{L^2(\Gamma)},
\end{equation}
and
\begin{equation}\label{est.2}
\Big\| \int _0^t e^{-i(t-s)}P {\bf{F}} (s)ds\Big \|_{L^q((0,T),L^r(\Gamma))}\leq C\|{\bf{F}} \|_{L^{\tilde r'}((0,T), L^{\tilde r'}(\Gamma))}.
\end{equation}
Using now Stone's theorem we obtain that
$$e^{-itH}\phi=e^{-itH}P\phi+e^{-itH}P_e\phi=e^{-itH}P\phi+\sum _{k=1}^m e^{it\lambda_k^2} <\phi,\varphi_k>\varphi_k,$$
where by $\lambda_k$ we denote the eigenvalue of the eigenfunction $\varphi_k$. 
We claim that for all $\alpha\geq 1$,
\begin{equation}\label{est.3}
\|e^{-itH}P_e\bu_0\|_{L^q((0,T),L^r(\Gamma))}\leq C T^{1/q} \|\bu_0\|_{L^\alpha(\Gamma)}
\end{equation}
and
\begin{equation}\label{est.4}
\Big\| \int _0^t e^{-i(t-s)}P_e {\bf{F}} (s)ds\Big \|_{L^q((0,T),L^r(\Gamma))}\leq CT^{1/q}\|{\bf{F}} \|_{L^1((0,T), L^{\alpha}(\Gamma))}.
\end{equation}
Putting together estimates \eqref{est.1}, \eqref{est.2}, \eqref{est.3} and \eqref{est.4} we obtain the desired result. We now prove estimates 
\eqref{est.3} and \eqref{est.4}.

In the case of estimate \eqref{est.3} using the fact   that
$$e^{-itH}P_e\bu_0=\sum _{k=1}^m e^{it\lambda_k^2} <\bu_0,\varphi_k>\varphi_k$$
we obtain by H\"older's inequality that for any $\alpha\geq 1$,
\begin{align*}
\|e^{-itH}P_e\bu_0\|_{L^r(\Gamma)}&\leq \sum _{k=1}^m | <\bu_0,\varphi_k>| \|\varphi_k\|_{L^r(\Gamma)}\\
&\leq 
\|\bu_0\|_{L^\alpha(\Gamma)} \sum _{k=1}^m \|\varphi_k\|_{L^{\alpha'}(\Gamma)} \|\varphi_k\|_{L^r(\Gamma)}\leq C\|\bu_0\|_{L^\alpha(\Gamma)}.
\end{align*}
Taking the $L^q$-norm on the time interval $(0,T)$ we obtain estimate \eqref{est.3}.

In a similar way we have
$$\|e^{-i(t-s)H}P_e{\bf{F}}(s)\|_{L^r(\Gamma)}\leq C\|{\bf{F}}(s)\|_{L^{\alpha}(\Gamma)}.$$
Using Minkowski's inequality we obtain that
\begin{align*}
\Big \| \int _0^t e^{-i(t-s)H}P_e {\bf{F}}(s)ds\Big \|_{L^q((0,T),L^r(\Gamma))}&\leq \Big 
\| \int _0^t \| e^{-i(t-s)H}P_e{\bf{F}}(s)\|_{L^r(\Gamma)}ds \Big\|_{L^q(0,T)}\\
&\leq T^{1/q}\int _0^T \|{\bf{F}}(s)\|_{L^\alpha(\Gamma)}ds.
\end{align*} 
which proves estimate \eqref{est.4}.

\end{proof}

\section{Appendix: general couplings}\label{gen}
In this appendix we consider general coupling conditions at each vertex $v$ (see \eqref{con.1} in \S 2),
$$
A^v{ \bf {u}} (v)+B^v{\bf {u}}'(v)=0.
$$
Using the notations introduced in this article, we shall give the recursion formulae for obtaining $\det D_{\Gamma_p}$ for general couplings. As a consequence, we shall give a sufficient condition for obtaining the dispersion.

We follow the approach in \S\ref{sect:res} for computing the resolvent.  
For a star-shaped graph with $n_1$ edges $I_j$ parametrized by $x\in[0,\infty[$, with coupling conditions $(A^1,B^1)$, the resolvent on each edge $I_j$ is
$$
R_\omega \bold u_0(x)=\tilde c_j\,e^{-\omega
x}+\frac{1}{2\omega}\int_0^\infty \bold u_0(y)\,e^{-\omega \mo{x-y}}dy.$$
The coupling conditions yield as a system for $\tilde c$'s:
$$(A^1+\omega  B^1)\left(\begin{array}{c}\tilde c_1\\:\\\tilde c_{n_1} \end{array}\right)=\left(\begin{array}{cccccccc}
\sum_{1\leq j\leq n_1}\frac{t_j(0,\omega)}{\omega}(b_{1,j}\,\omega-a_{1,j})\\
...\\
\sum_{1\leq j\leq n_1}\frac{t_j(0,\omega)}{\omega}(b_{n_1,j}\,\omega-a_{n_1,j}) \end{array}\right).$$
We denote by $D_{\Gamma_1}(\omega)$ the matrix of the system. We define $\tilde D_{\Gamma_1}(\omega)$ to be the matrix $((A^1+\omega B^1)_1,(A^1+\omega B^1)_2,...,(A^1-\omega B^1)_{n_1})$. By $(A^1+\omega B^1)_j$ we mean the $j$th column of $A^1+\omega B^1$.
 
The case of a general tree with $p$ vertices can again be seen as constructed by adding a new vertex $v^p$ to a $p-1$-vertices tree, with coupling conditions $(A^p,B^p)$, from which emerge new $n_p-1$ infinite edges. Similarly as for Lemma \ref{dets} we derive the recursion formulae
$$\det D_{\Gamma_1}(\omega)=\det(A^1+\omega  B^1),$$
$$\frac{\det \tilde D_{\Gamma_{1}}(\omega)}{\det D_{\Gamma_{1}}(\omega)}=\frac{\det((A^1+\omega B^1)_1,(A^1+\omega B^1)_2,...,(A^1-\omega B^1)_{n_1})}{\det(A^1+\omega B^1)},$$
$$\det D_{\Gamma_p}(\omega)=\det(A^p+\omega B^p)\,e^{\omega a_{p-1}}\det D_{\Gamma_{p-1}}(\omega)$$
$$\times\left(1-\frac{\det((A^p-\omega B^p)_1,(A^p+\omega B^p)_2,...,(A^p+\omega B^p)_{n_p})}{\det(A^p+\omega B^p)}\,e^{-2\omega a_{p-1}}\frac{\det \tilde D_{\Gamma_{p-1}}(\omega)}{\det D_{\Gamma_{p-1}}(\omega)}\right).$$
$$\frac{\det \tilde D_{\Gamma_{p}}(\omega)}{\det D_{\Gamma_{p}}(\omega)}=\left(\frac{\det((A^p+\omega B^p)_1,(A^p+\omega B^p)_2,...,(A^p-\omega B^p)_{n_p})}{\det(A^p+\omega B^p)}\right.$$
$$\left.-\frac{\det((A^p-\omega B^p)_1,(A^p+\omega B^p)_2,...,(A^p-\omega B^p)_{n_p})}{\det(A^p+\omega B^p)}\,e^{-2\omega a_{p-1}}\frac{\det \tilde D_{\Gamma_{p-1}}(\omega)}{\det D_{\Gamma_{p-1}}(\omega)}\right)$$
$$\times\left(1-\frac{\det((A^p-\omega B^p)_1,(A^p+\omega B^p)_2,...,(A^p+\omega B^p)_{n_p})}{\det(A^p+\omega B^p)}\,e^{-2\omega a_{p-1}}\frac{\det \tilde D_{\Gamma_{p-1}}(\omega)}{\det D_{\Gamma_{p-1}}(\omega)}\right)^{-1}.$$ 
A sufficient condition for using the spectral formula as in \S \ref{sect:end} and then for getting the dispersion the following constraint, depending only on the entries of $(A^j,B^j)_{1\leq j\leq p}$ is the following one:
\begin{equation}\label{conddet}
|\det D_{\Gamma_p}(i\omega)|\neq 0,\forall \omega\in\mathbb R.
\end{equation}
This is the way the Kirchhoff coupling case was ruled in \cite{MR2858075} and this might be used in other cases. In the $\delta-$coupling case presented in this article, and probably in many other cases, such an estimate is not valid. Then an analysis around the  zeros of $\det D_{\Gamma_p}(\omega)$ has to be done starting from the above recursion formulae.  


\end{document}